\documentclass[12pt]{amsart}

\usepackage{a4wide,amsfonts,amsmath,amsthm}
\usepackage{latexsym,multirow}

\newtheorem{theorem}{Theorem}[section]
\newtheorem{corollary}[theorem]{Corollary}

\newtheorem{proposition}[theorem]{Proposition}

\theoremstyle{definition}

\newcommand{\Co}{\ensuremath{\mathbb{C}}}

\newcommand{\smc}{\ensuremath{\mathcal{S}}}

\def \< {\langle}
\def \> {\rangle}

\begin{document}

\title[Character tables and finite projective planes]{Character tables and the problem of existence of finite projective planes}

\author[M. Matolcsi]{M\'at\'e Matolcsi}
\address{M.M.: Budapest University of Technology and Economics (BME),
H-1111, Egry J. u. 1, Budapest, Hungary (also at Alfr\'ed R\'enyi Institute of Mathematics,
Hungarian Academy of Sciences, H-1053, Realtanoda u 13-15, Budapest, Hungary)}
\email{matomate@renyi.hu}

\author[M. Weiner]{Mih\'aly Weiner}
\address{M.W.: Budapest University of Technology and Economics (BME),
H-1111, Egry J. u. 1, Budapest, Hungary}
\email{mweiner@math.bme.hu}

\thanks{M. Matolcsi was supported by the ERC-AdG 321104, M. Weiner was supported by the ERC-AdG 669240 QUEST ``Quantum Algebraic Structures and Models'' and by OTKA Grant NKFI K 124152}


\begin{abstract}
Recently, the authors of the present work (together with M.\! N.\! Kolountzakis) introduced a new version of the non-commutative Delsarte scheme and applied it to the problem of mutually unbiased bases. Here we use this method to investigate the existence of a finite projective plane of a given order $d$. In particular, 
a short new proof is obtained for the nonexistence of a projective plane of order 6. For higher orders like 10 and 12, the method is non decisive but could turn out to give important supplementary informations.
\end{abstract}

\maketitle

\bigskip



\section{Introduction}

Given a group $G$ and a ``forbidden set'' $A=A^{-1}\subset G$, at most how many elements a subset $B=\{b_1,\ldots b_n\}\subset G$ can have, if all ``differences''$b_j^{-1}b_k$ ($j\neq k$) ``avoid'' $A$? This is a very general type of question; many famous problems can be re-phrased in this manner.

A method that often proved to be fruitful when dealing with such problems 
--- e.g.\! in the context of sphere-packing \cite{KabLev,cohnelkies} or in the maximum number of code-words \cite{delsarte} ---
is the Fourier-analytical approach first pioneered by Delsarte \cite{delsarte} in the commutative case. Later Oliveira Filho and Vallentin proved the optimization bound \cite[Theorem 2]{vallentinfilho}, which can be viewed as a generalization of Delsarte's method that includes the non-commutative case. Recently, together with Kolountzakis, the authors of this work presented another non-commutative Delsarte-scheme \cite{mi}, which they applied to the famous quantum physically motivated question regarding the existence of complete collections of mutually unbiased bases in $\Co^d$. Since there are many indications pointing to some relation between this question and that of the existence of a projective plane of order $d$, it seems natural to also re-phrase and investigate this latter problem in the above manner.

For any prime-power order, at least one projective plane can be constructed using finite fields. However, although there are also other constructions, so far no one has managed to present a projective plane with a non prime-power order and in fact it is widely believed that there are no such planes. In the beginning of the XX.\! century, Tarry \cite{tarry} proved that there is no $6\times 6$ Greco-Latin square, which is actually a stronger statement than the nonexistence of a finite projective plane of order $6$. However, his proof is not so instructing, as it is based on a rather tedious hand-checking of each $6\times 6$ Latin square. Then some 40 years later, Bruck and Ryser \cite{bruckryser} proved that if a finite projective plane of order $d \equiv 1,2$ mod$(4)$ exists, then $d$ must be a sum of two squares. This again rules out the existence of a finite projective plane of order $6$ (and of course many other orders, too), but leaves open the question for the next two non prime-powers: $d=10$ and $d=12$. For $d=10$ we only know the nonexistence because of a massive computer search \cite{LamThiefSwiercz}, and for $d=12$, the question is still open.

Although the argument of Bruck and Ryser is elegant, their proof is certainly not one of those that can be digested in 5 minutes. Thus, up to now, even for the smallest non prime-power order --- i.e.\! for $d=6$ --- there was no easily presentable, short argument for the nonexistence in question; hence we felt that there was room here for yet another proof.

Let us shortly see now how the existence of projective planes of a given order $d$ can be investigated in a Delsarte scheme. Instead of finite projective planes, one may work with some equivalent structures like that of finite affine planes or complete sets of mutually orthogonal Latin squares. For our purpose we shall depart from a finite affine plane of order $d$. We fix and enumerate the lines of two of its parallel equivalence classes so that we have a ``coordinate system'' on our plane. As any further line intersects each ``horizontal'' and ``vertical'' line exactly once, we can view each such line as the graph of a bijective function $\{1,2\ldots d\}\to \{1,2\ldots d\}$; that is, an element of the permutation group $S_d$. In this way, the remaining $(d-1)d$ lines of the affine plane are encoded in $(d-1)d$ permutations $\sigma_1,\sigma_2,\ldots \sigma_{(d-1)d}\in S_d$. Then $\sigma_j$ and $\sigma_k$ stand for two (different) parallel (i.e.\! non-intersecting) lines if and only if $\sigma_j^{-1}\sigma_k$ has no fixed points, whereas if they are from different parallel classes, then --- as they must intersect in exactly one point --- $\sigma_j^{-1}\sigma_k$ must have precisely one fixed point.
So one may ask: at most how large a subset $B\subset S_d$ can be so that the difference $\sigma^{-1}\tilde{\sigma}$ between any two different elements $\sigma,\tilde{\sigma}\in B$ ``avoids'' the forbidden set $A$ consisting of permutations with more than one fixed points? As we have seen, if this number is less then $(d-1)d$, then there can be no projective plane of order $d$. 

Note that actually we have more information than just the fact that the differences must avoid the forbidden set $A$. If we assume that $B=\{\sigma_1,\ldots \sigma_{(d-1)d}\}\subset S_d$ comes from an affine plane of order $d$, then, apart from saying that the differences are not in $A$, we can even tell how many of the differences have one fixed point and how many of them have zero. In our proof of nonexistence, we shall also make use of such finer details.

The Delsarte scheme recently presented in \cite{mi} involves general positive (semi)definite functions on $G$. Now in our present case --- in contrast with the cited work --- the forbidden set $A$ is invariant under conjugations. As a consequence, it  suffices to consider positive semidefinite functions that are also {\it class-functions}; i.e.\! ones that take constant values inside each conjugacy class. However, using representation theory, it is a rather easy exercise to show that if $G$ is a finite group and $h:G\to \Co$ is a class-function, then $h$ is positive semidefinite if and only if it is a linear combination of the irreducible characters of $G$ with nonnegative coefficients only. For this reason --- and also for self-containment --- rather than recalling the mentioned recently introduced general method (involving positive semidefinite functions), here we shall give a presentation directly formulated in terms of class-functions and characters.  

Assuming that $B=\{\sigma_1,\ldots \sigma_{(d-1)d}\}\subset S_d$ comes from an affine plane of order $d$, --- as we shall see in the next section --- our method gives a system of linear equations and inequalities regarding the number of differences (between different elements of $B$) falling in each conjugacy class of $S_d$.
However, in general one can easily find some further conditions that are somehow not ``recognized'' by our method; e.g.\! that these numbers must be all even integers. 
So, although for $d=6$ our linear system does have a --- unique  --- solution, we could then easily conclude the nonexistence of a finite projective plane of order $6$.

This can all be done on a small piece of paper because of two reasons. First, because the character table of $S_d$ is well known (described by the the so-called {\it Murnaghan-Nakayama rule}; see e.g.\! in the book \cite{book}), second, because in the $d=6$ case, 
$S_d$ has only $11$ conjugacy classes and hence we have a rather small linear system  to solve. However, for $d=12$ for example, $S_d$ has already $77$ conjugacy classes and a similar computation by hand would be extremely cumbersome. Nevertheless, using a computer it easy to solve a linear programming problem even with hundreds of variables and equations / inequalities. 

So we went ahead and tried out what happens up to $d=12$. We found that up to $d=6$ there is a unique solution, but uniqueness breaks down starting from $d=7$ --- even though that up to equivalence, there is a unique projective plane of order seven \cite{pierce,hall}. Like for $d=6$, we also tried for $d=10$ and $12$ to add some further conditions using {\it ad hoc} considerations. However, unlike in the $d=6$ case, these were not sufficient to arrive to a contradiction. We still hope though that the information given by our method will turn out to be useful in the future even at higher orders.
 
\section{The method applied}

Let $G$ be a finite group with conjugacy classes $C_0=\{e\},C_1\ldots C_r$ and let 
$\gamma$ be the function assigning to each element the cardinality of the conjugacy class it is contained in:
$$
\gamma|_{C_k}\, =\, | C_k | \;\; (k=0,\ldots r).
$$
For a $B=\{b_1,\ldots b_n\}\subset G$ we shall consider the class-function $\theta_B$ counting the number of times that the difference between elements of $B$ falls in a certain conjugacy class; that is,
$$
\theta_B|_{C_k} \, = \,|\{(j,m)|b_j^{-1}b_m\in C_k\}| \;\; (k=0,\ldots r).
$$
Note that $\theta_B$ takes nonnegative values (actually: nonnegative {\it integer} values) only, $\theta_B(e)= |B| = n$ and as there are $n^2$ differences altogether, we also have that $\sum_{g\in G}\frac{\theta_B(g)}{\gamma(g)}= |B|^2=n^2$.
Apart from these obvious ones, our main observation is the following.
\begin{proposition}
For any character $\chi$ of $G$, the value of the scalar product 
$$
\langle \chi, (\theta_B/\gamma) \rangle \, \equiv\, \frac{1}{|G|}\sum_{g\in G} \overline{\chi(g)}\frac{\theta_B(g)}{\gamma(g)}
$$
between $\chi$ of $G$ and the class-function $\theta_B/\gamma$ is a nonnegative real. Hence $\theta_B/\gamma$  is a linear combination of the irreducible characters of $G$ with nonnegative coefficients only.
\end{proposition}
\begin{proof}
It is clearly enough to prove the first affirmation; as the irreducible characters form an orthonormed bases in the space of class-functions, the value of the scalar product in question is precisely the coefficient of $\theta_B/\gamma$ in this basis corresponding to $\chi$. Now let $U$ be the representation giving $\chi$. As $G$ is finite, we may safely assume that $U$ is actually a unitary representation. Then setting $X=\sum_{j=1}^n U(b_j)$, we have
\begin{eqnarray}
\nonumber
|G|\, \overline{\langle \chi, (\theta_B/\gamma) \rangle}  &=& \sum_{g\in G} \chi(g)\frac{\theta_B(g)}{\gamma(g)} =\sum_{k=0}^r 
\left(\chi\, \theta_B\right)|_{C_k} = 
\sum_{k=0}^r \chi|_{C_k} \, |\{(j,k)|b_j^{-1}b_k\in C_k\}|\\
\nonumber
&=& \sum_{j,m=1}^n \chi(b_j^{-1}b_m) = \sum_{j,m=1}^n {\rm Tr}(U(b_j^{-1}b_m))
\\ 
&=&
\sum_{j,m=1}^n {\rm Tr}(U(b_j)^* U(b_m)) = {\rm Tr}\left(X^*X\right)\geq 0,
\end{eqnarray}
showing the non-negativity of the scalar product in question.
\end{proof}

Now let us consider the case when $G=S_d$ and the subset $B=\{\sigma_1,\ldots\sigma_{(d-1)d}\}$ is given by an affine plane of order 
$d$ as explained in the introduction. Then of the total of $((d-1)d)^2$ differences
between the elements of $B$, $(d-1)d$ gives the identity (i.e.\! a difference between an element and itself), $(d-1)^2d$ are differences between elements corresponding to two (different) lines of the same parallel class and $(d-2)(d-1)d^2$ are differences
between elements corresponding lines of different equivalent classes. Thus, denoting by $\smc_{j}$ the collection of conjugacy classes of $S_d$ containing permutations with $j=0,1,\ldots d$ fixed points, we have the linear equations
\begin{equation}
\label{equalities}
\left\{
\begin{array}{rcc}
\forall C \notin (\{e\}\cup \smc_0\cup \smc_1):\;\;\;\theta_B|_C &=& 0,
\vspace{1mm}
\\
\theta_B(e) & = & (d-1)d ,
\vspace{1mm}
\\
\displaystyle\sum_{C\in \smc_0}\theta_B|_C &=& (d-1)^2d, 
\\
\displaystyle\sum_{C\in \smc_1}\theta_B|_C&=&(d-2)(d-1)d^2
\end{array}
\right.
\end{equation}
as well the linear inequalities given by our previous proposition and the noted fact that $B$ takes nonnegative values only:
\begin{equation}
\label{inequalities}
\left\{
\begin{array}{rcc}
\forall C:\;\;\;\theta_B|_C &\geq & 0, \\
\forall \chi \,{\textrm{irr. char.}}:\;\;\; \displaystyle\sum_{C} (\chi \theta_B)|_C &\geq & 0.
\end{array}
\right.
\end{equation}
We view this linear system as a restriction on possible $\theta_B$ functions. (Note that we dropped the conjugation signs as both $\theta_B$ and $\chi$ are real-valued functions\footnote{A particular ``feature'' of the permutation group is that all of its characters are real-valued.}.) 

Let us consider now in particular the case when $d=6$. $S_6$ has $11$ conjugacy classes, of which $2$ are in $\smc_1$ and $4$ in $\smc_0$. Since we do not need permutations with $2,3$ or $4$ fixed points (as $\theta_B$ is constant zero over them), the following shortened version of the character table of $S_6$ will be sufficient for us:

\begin{center}
$\hphantom{aaaaaa}
\overbrace{\hphantom{aaaaaaaaaaaaaaaaaaaaaaaaaaaaaaaaaaaaaaa}}^{\smc_0}
\hphantom{aa}
\overbrace{\hphantom{aaaaaaaaaaaaaaa}}^{\smc_1}$
\begin{tabular}{r||c||c|c|c|c||c|c}
& e & (123)(456) & (12)(34)(56) & (1234)(56) & (123456) & (123)(45) & (12345) \\
\hline
$\chi_1$ & 1 & 1 & 1 & 1 & 1 & 1 & 1 \\
$\chi_2$ & 1 & 1 & -1 & 1 & -1 & -1 & 1 \\
$\chi_3$ & 5 & -1 & -1 & -1 & -1 & 0 & 0 \\
$\chi_4$ & 5 & -1 & 1 & -1 & 1 & 0 & 0 \\
$\chi_5$ & 5 & 2 & 3 & -1 & 0 & -1 & 0 \\
$\chi_6$ & 5 & 2 & -3 & -1 & 0 & 1 & 0 \\
$\chi_7$ & 9 & 0 & 3 & 1 & 0 & 0 & -1 \\
$\chi_8$ & 9 & 0 & -3 & 1 & 0 & 0 & -1 \\
$\chi_9$ & 10 & 1 & 2 & 0 & -1 & 1 & 0 \\
$\chi_{10}$ & 10 & 1 & -2 & 0 & 1 & -1 & 0 \\
$\chi_{11}$ & 16 & -2 & 0 & 0 & 0 & 0 & 1 
\\
\hline
$\theta_B$ & 30 & $x$ & $y$ & $z$ & $v$ & $a$ & $b$ 
\end{tabular}
$\hphantom{aaaaaa}
\underbrace{\hphantom{aaaaaaaaaaaaaaaaaaaaaaaaaaaaaaaaaaaaaaa}}_{150}
\hphantom{aa}
\underbrace{\hphantom{aaaaaaaaaaaaaaa}}_{720}$
\end{center}
In the last line we put (as parameters) the values of the function $\theta_B$, already indicating the system (\ref{equalities}) of equalities; namely, that $\theta_B(e)=5*6=30$, $x+y+z+v=5^2*6=150$ and $a+b=4*5*6^2=720$. What remains is to make use of the inequalities (\ref{inequalities}), which tell us that all parameters 
$x,y,z,v,a,b$ as well as the scalar product of the last line of the table (corresponding to $\theta_B$) with any other line is nonnegative.

In particular, considering that the sum of the lines corresponding to the characters $\chi_5,\chi_7,\chi_8$ and $\chi_{10}$ is $(33,3,1,1,1,-1,-1)$, we have the inequality:
\begin{equation}
33*30+3x+y+z+v-2a-2b = 990+2x +(x+y+z+v)-2(a+b)\geq 0.
\end{equation}
Thus, as $x+y+z+v=150$ and $a+b=720$, we
have that $990+2x+150-2*720\geq 0$ which results in $x\geq 150$. On the other hand,
$x$ is {\it at most} $150$, as $x+y+z+v=150$. Hence we must have 
$x=150$ and $y=z=v=0$.
Then the scalar product with the lines corresponding to $\chi_5$ and $\chi_7$ can be simplified resulting the inequalities
\begin{equation}
5*30+2*150-a\geq 0,\;\; 9*30-b\geq 0.
\end{equation}
Hence $a\leq 450$ and $b\leq 270$. But then we have no choice: these inequalities must actually be equalities, as $a$ and $b$ must add up to $720$. Thus the only possibility is:
\begin{equation}
\label{solution}
x=150, y=0, y=0, v=0, a=450,b=270.
\end{equation}
At this point we could just stop the investigation of our linear system, since in what follows we shall only use that this is how $\theta_B$ {\it should be}. However, we note that simple check shows that with the values given by (\ref{solution}), all scalar products are indeed positive and hence (\ref{solution}) is actually a solution.
\begin{corollary}
There exists no finite projective plane of order $6$.
\end{corollary} 
\begin{proof}
Assume that there exists a projective --- and hence also an affine --- plane of order $d=6$. Then, as explained there should exists a collection of $5*6=30$ permutations $B=\{\sigma_1,\ldots \sigma_{30}\}\subset S_6$ describing the lines of the ``last'' $d-1=5$ parallel classes of our affine plane, with corresponding ``difference-counting'' function $\theta_B$ given by (\ref{solution}). In particular, 
out of the total of $900$ differences, $450$ should be of negative parity. This is only possible, if half of our permutations (i.e.\! $15$ out of the $30$) are of positive, and half are of negative parity (forming $30*15=450$ ordered pairs with opposing signs); in any other case there would be fewer differences of negative parity. However, as $y=v=0$, all differences with zero fixed points are of positive parity and hence the permutations corresponding to the lines of a single parallel class should have the same sign. Thus, the number of elements in $B$ with positive parity should be divisible by $6$ (as each parallel class contains $6$ lines), which contradicts to what we established earlier; namely, that precisely $15$ of the elements of $B$ should have positive parity. 
\end{proof}

\end{document}